\documentclass[a4paper,10pt]{amsart}

\usepackage[english]{babel}
\usepackage[utf8]{inputenc}
\usepackage[T1]{fontenc}
\usepackage{ulem} 

\usepackage{amssymb}
\usepackage{amsmath}
\usepackage{amsthm}
\usepackage{bbm}

\usepackage{enumerate}
\usepackage{tikz}
\usepackage{hyperref}
\usepackage{marginnote}
\usepackage{mathtools}
\usepackage{enumitem}
\usepackage{orcidlink}

\newcommand{\bbC}{\mathbb{C}}

\newcommand{\bbN}{\mathbb{N}}

\newcommand{\bbR}{\mathbb{R}}

\newcommand{\calA}{\mathcal{A}}
\newcommand{\calB}{\mathcal{B}}

\newcommand{\calE}{\mathcal{E}}

\newcommand{\calL}{\mathcal{L}}

\newcommand{\R}{\bbR}
\newcommand{\C}{\bbC}

\DeclarePairedDelimiter{\norm}{\lVert}{\rVert}
\DeclarePairedDelimiter{\abs}{\lvert}{\rvert}

\DeclarePairedDelimiter{\set}{\lbrace}{\rbrace}

\DeclareMathOperator{\id}{id}
\DeclareMathOperator{\one}{\mathbbm{1}}

\newcommand{\argument}{\mathord{\,\cdot\,}}

\newcommand{\ue}{\mathrm{e}}

\newcommand{\ud}{\mathrm{d}}

\theoremstyle{definition}
\newtheorem{definition}{Definition}[section]

\theoremstyle{plain}
\newtheorem{proposition}[definition]{Proposition}
\newtheorem{lemma}[definition]{Lemma}
\newtheorem{theorem}[definition]{Theorem}
\newtheorem{corollary}[definition]{Corollary}

\numberwithin{equation}{section}

\begin{document}

\normalem

\title[Uniform convergence for stochastic hybrid models]{Uniform convergence of solutions to stochastic hybrid models of gene regulatory networks}
\author[Alexander Dobrick]{Alexander Dobrick
\orcidlink{0000-0002-3308-3581}
}
\address[A.\ Dobrick]{Arbeitsbereich Analysis, Christian-Albrechts-Universit\"at zu Kiel, \linebreak Heinrich-Hecht-Platz 6, 24118 Kiel, Germany}
\email{dobrick@math.uni-kiel.de}

\author[Julian Hölz]{Julian Hölz
\orcidlink{0000-0001-5058-9210}
}
\address[J.\ Hölz]{Bergische Universität Wuppertal, Fakultät für Mathematik und Naturwissenschaften, Gaußstr.\ 20, 42119 Wuppertal, Germany}
\email{hoelz@uni-wuppertal.de}

\date{\today}
\begin{abstract}
    In a recent paper by Kurasov, Lück, Mugnolo and Wolf, a hybrid gene regulatory network was proposed to model gene expression dynamics by using a stochastic system of coupled partial differential equations. In more recent work, the existence and strong convergence of the solutions to equilibrium was proven. In this article, we improve upon their result by showing that the convergence rate is independent of the initial state, therefore proving that the solutions converge not only strongly but even uniformly to equilibrium. To this end, we make use of a recent convergence theorem for stochastic, irreducible semigroups that contain partial integral operators. 
\end{abstract}

\maketitle

\section{Introduction}

Let $a, b, c, d \in \bbR$ such that $b, d > 0$ and $\tfrac{a}{b} < \tfrac{c}{d}$. Set $I \coloneqq [\frac a b, \frac c d]$ and let $\nu, \mu: I \to (0, \infty)$ be continuous. Given an initial state $u_0 \in L^1(I; \bbC^2)$, we are interested in the long-term behaviour of the solutions to the system
\begin{equation} \label{eq:the-pde}
    \left\lbrace
    \begin{aligned}
        \frac {\ud} {\ud t} u(t,x) &= - \frac {\ud} {\ud x}
        \begin{pmatrix}
            (a - b x) u_1(t,x) \\
            (c - d x) u_2(t,x)
        \end{pmatrix}
        +
        \begin{pmatrix}
    		-\nu(x) & \phantom{-} \mu(x) \\
    		\phantom{-} \nu(x) & -\mu(x)
	    \end{pmatrix}
	     \begin{pmatrix}
    		u_1(t, x) \\
    		u_2(t, x)
	    \end{pmatrix}, \\
        u(0, x) &= u_0(x),
    \end{aligned} 
    \right. 
\end{equation}
subject to the boundary conditions
\begin{align*}
    u_1 \bigl(t, \tfrac c d \bigr) = 0 \quad \text{and} \quad u_2 \bigl(t, \tfrac a b \bigr)  = 0 \qquad \text{for all } x \in I \text{ and } t \geq 0.
\end{align*}
In \cite{Kurasov2018}, the authors Kurasov, Lück, Mugnolo and Wolf described a hybrid gene regulatory network model aiming to model the evolution of protein counts in cells on a continuous state space by a piecewise deterministic Markov processes, i.e., a system of coupled partial differential equations. Classically, the state space of discrete-stochastic models representing gene regulatory networks is typically high-dimensional due to the combinatorial nature of molecule counts (see \cite{Schnoerr2017} for a survey). However, under additional assumptions, exact solutions of the discrete state space models can be found (see, e.g., \cite{Jahnke2006, Laurenzi2000}). Due to their high-dimensional nature, these models are difficult to solve numerically. Hybrid models of gene regulatory networks emerged as a compromise between the accuracy of deterministic models and the computational efficiency of stochastic models (see \cite{Puchalka2004}). The model in~\eqref{eq:the-pde} follows the same approach and models gene expressions in single cells on a continuous state space. In \cite{Kurasov2021} exact solutions of~\eqref{eq:the-pde} for special choices of parameters are given and proven to converge individually to equilibrium.

Following \cite{Kurasov2018, Kurasov2021}, we discuss the long-term behaviour of solutions of \eqref{eq:the-pde} by using classical semigroup methods. In \cite[Proposition~4.3]{Kurasov2021} it was proven that the problem is well-posed and that its solution is given by a stochastic, irreducible semigroup when modelled on a suitable $L^1$-space. Moreover, in \cite[Theorem~4.6]{Kurasov2021} the authors prove the existence of equilibrium states and that the solutions to \eqref{eq:the-pde} converge strongly to equilibrium.

In this paper, we strengthen \cite[Theorem~4.6]{Kurasov2021} to operator norm rather than strong convergence. To be precise, we prove the following result.

\begin{theorem}\label{theorem:main-result}
    The solution semigroup $(S(t))_{t \geq 0}$ of~\eqref{eq:the-pde} converges to a rank-$1$-projection onto its fixed space with respect to the operator norm as $t \to \infty$.
\end{theorem}

In other words, the convergence of the solutions to the equilibrium cannot happen arbitrarily slow, while varying over all the initial conditions of the system. A few comments on the above result are in order:
\begin{enumerate}[label = \upshape (\alph*)]
    \item The first order differential operator occurring in~\eqref{eq:the-pde} essentially acts as a shift; in particular, it exhibits no smoothing on the solutions unlike diffusion terms. Hereby, the entire mass in the different coordinates is shifted in two opposing directions and concentrated at the opposing ends of the interval $I$. Hence, one cannot expect the solutions to converge without the action of the potential term. 

    \item The potential term is a multiplication operator and, therefore, has also no particular smoothing properties. However, the potential term causes a mixing of the coordinates resulting in these shifts to counteract each other, ensuring the convergence of solutions. This idea was the gist in \cite{Kurasov2021}.
\end{enumerate}

\subsection*{Methodology and related literature}
There exists a substantial number of results in the literature on the long-term behaviour of semigroups on $L^p$-spaces that show strong convergence to an equilibrium as $t \to \infty$ under the assumption that the semigroup contains an integral operator (see Section~\ref{section:the-long-term-behaviour} for a definition). These results have a long history, going back to \cite[Korollar~3.11]{Greiner1982}. Subsequent works provide similar results (see \cite{Arendt2008, Davies2005, Gerlach2013, GerlachDISS, Gerlach2017} for examples). For a brief survey of these results, we refer to \cite[Section~3]{Arendt2020}. Furthermore, in \cite{Gerlach2019} results of this type are obtained under the condition that the semigroup contains an operator that dominates a non-zero integral operator. In \cite{Glueck2019} a generalization of many of the aforementioned results is presented in a more abstract setting. 

Furthermore, in \cite{Pichor2000, Pichor2016, Pichor2018a, Pichor2018, Rudnicki2002} and \cite[\S 5]{Pichor2012} an asymptotic theory of substochastic $C_0$-semigroups that dominate an integral operator on an $L^1$-space is developed. These results all yield the strong convergence of semigroups to equilibrium and have a history of applications to models from mathematical biology, kinetic equations and queueing theory. In the proof of Theorem~\ref{theorem:main-result} we employ a more recent theorem from~\cite{Glueck2022} that yields the stronger operator norm convergence of stochastic semigroups (see Theorem~\ref{theorem:glueck-martin}).

\subsection*{Organization of the article}
In Section~\ref{section:setting-the-stage} we study the first-order differential operators that appear in~\eqref{eq:the-pde}. Moreover, the proof of Theorem~\ref{theorem:main-result} is prepared by simplifying~\eqref{eq:the-pde} via a state space transform. This shortens the calculations in Section~\ref{section:the-long-term-behaviour}.

In Section~\ref{section:the-long-term-behaviour} the main result is proved. This is achieved by showing the existence of a lower bound for the dual resolvent of the generator of the solution semigroup $(T(t))_{t \geq 0}$ and by proving that the solution semigroup contains a partial integral operator. This allows us to apply the aforementioned result from \cite{Glueck2022}.

\subsection*{Notation and terminology} 
Let $E$ be a Banach space. Then we denote the \emph{space of bounded linear operators} on $E$ by $\calL(E)$. Furthermore, if $A$ is a closed operator on $E$, then we denote by $\sigma(A)$ the \emph{spectrum} and by $\rho(A)$ the \emph{resolvent set} of $A$.  Let $(\Omega, \Sigma, \mu)$ be a measure space. For an element $f \in L^\infty(\Omega)$, we use the notation $f \gg0$ to say that $f$ is \emph{bounded away from $0$}, i.e., that there exists some $c > 0$ such that $f \geq c \cdot \one$, where $\one \colon \Omega \to \C, \, x \mapsto 1$. Here, the inequality $\geq$ between elements in $L^\infty(\Omega)$ is to be understood pointwise for almost all points in $\Omega$.

Furthermore, a $C_0$-semigroup $(T(t))_{t \geq 0}$ on $L^1(\Omega)$ is called \emph{positive} if $T(t) f \geq 0$ for all $0 \leq f \in L^1(\Omega)$. A positive $C_0$-semigroup $(T(t))_{t \geq 0}$ on $L^1(\Omega)$ is called \emph{stochastic} if
\begin{align*}
    \norm{T(t) f} = \norm{f} \qquad \text{for all } 0 \leq f \in L^1(\Omega). 
\end{align*}

A $\C^n$-valued function $f$ is said to satisfy the above conditions if it satisfies them componentwise.

Throughout the paper, we assume knowledge of classical semigroup theory and the basics of the theory of positive operators on $L^1$-spaces. For an introduction to these topics, we refer to \cite{Batkai2017, Engel2000, Nagel1986, Rudnicki2017} and \cite{Aliprantis2006, Aliprantis1985, MeyerNieberg1991, Schaefer1974, Zaanen1997}, respectively.

\section{Setting the stage} \label{section:setting-the-stage}

Let $a, b, c, d$ as well as the interval $I$ and the functions $\nu$ and $\mu$ be defined as in the introduction and let $\mathring{I} \coloneqq (\frac a b, \frac c d)$. Consider the operator $\calA \colon D(\calA) \to L^1(\Omega; \C^2)$ defined by
\begin{align}\label{def:differential-operator-general}
	\calA
	\begin{pmatrix}
		f_1 \\
		f_2
	\end{pmatrix}
	(x) \coloneqq -\frac{\ud}{\ud x}
	\begin{pmatrix}
		a - b x & 0 \\
		0 & c - d x
	\end{pmatrix}
	\begin{pmatrix}
		f_1 \\
		f_2
	\end{pmatrix}(x),
\end{align}
where
\begin{equation*}
    D(\calA) \coloneqq \left\{f \in L^1(I; \C^2) : \frac{\ud}{\ud x}
    \begin{pmatrix}
	(a - b x) f_1\\
	(c - d x) f_2
    \end{pmatrix}
    \in L^1(\mathring{I}; \C^2), \, f_1(\tfrac c d) = 0, \, f_2 (\tfrac a b) = 0
    \right\}.
\end{equation*}
Here $L^1(I; \C^2)$ is equipped with the norm
\begin{align*}
    \norm{f}_1 \coloneqq \norm{f_1}_1 + \norm{f_2}_1 \qquad \text{for all } f \in L^1(I; \C^2).
\end{align*}
This norm renders $L^1(I; \C^2)$ isometrically isomorphic to the $L^1$-space of scalar-valued functions over two disjoint copies of $I$. In particular, one can treat $L^1(I; \C^2)$ as a scalar-valued $L^1$-space. 

Consider the bounded operator $\calB \colon L^1(I; \C^2) \to L^1(I; \C^2)$ given by
\begin{align*}
	\calB
	\begin{pmatrix}
		f_1 \\
		f_2
	\end{pmatrix}(x) \coloneqq
	\begin{pmatrix}
		-\nu(x) & \mu(x) \\
		\nu(x) & -\mu(x)
	\end{pmatrix}
	\begin{pmatrix}
		f_1 (x) \\
		f_2 (x)
	\end{pmatrix}.
\end{align*}

Using the operators $\calA$ and $\calB$ we can restate~\eqref{eq:the-pde} as the abstract Cauchy problem
\begin{align} \label{eq:ACP} \tag{ACP}
    u'(t) = (\calA + \calB) u(t), \quad u(0) = u_0
\end{align}
on the state space $L^1(I; \C^2)$, where $D(\calA + \calB) \coloneqq D(\calA)$.

Via an affine linear state-space transform, we map the interval $I = [\frac a b, \frac c d]$ to the interval $[0,1]$. The transformed system is then of the form
\begin{equation}\label{eq:the-pde-simplified}
    \left\lbrace
    \begin{aligned}
        \frac {\ud} {\ud t} u(t,x) &= \frac {\ud} {\ud x}
        \begin{pmatrix}
            b x u_1(t,x) \\
            d (x - 1) u_2(t,x)
        \end{pmatrix}
        +
        \begin{pmatrix}
    		-\tilde\nu(x) & \tilde\mu(x) \\
    		\tilde\nu(x) & -\tilde\mu(x)
	    \end{pmatrix}
	     \begin{pmatrix}
    		u_1(t, x) \\
    		u_2(t, x)
	    \end{pmatrix}, \\
        u(0, x) &= u_0(x),
    \end{aligned} 
    \right. 
\end{equation}
subject to the boundary conditions $u_1(t,1) = 0$ and $u_2(t, 0) = 0$. Here $\tilde\nu$ and $\tilde{\mu}$ are positive and continuous functions on the interval $[0,1]$. It follows that we may choose $a = 0$ and $c = d$ without losing the generality offered by the original system~\eqref{eq:the-pde}.

In the simplified form~\eqref{eq:the-pde-simplified} two differential operators appear. Namely, the operators
\begin{align*}
    f \mapsto \frac{\ud}{\ud x}(x f) \qquad \text{and} \qquad f \mapsto \frac{\ud}{\ud x} \big((1 -x) f \big).
\end{align*}
We commence by studying these operators. It turns out that it is easiest to consider both operators on a domain in $L^1([0, \infty))$ and $L^1((- \infty, 0])$, respectively. We later restrict the operators to $L^1([0,1])$. To this end, we define
\begin{alignat*}{2}
    A_1 f &\coloneqq \frac{\ud}{\ud x}(x f), \quad && D(A_1) \coloneqq \set{f \in L^1([0, \infty)) : (x f)' \in L^1((0, \infty))}, \\
    A_2 f &\coloneqq \frac{\ud}{\ud x}((x - 1) f), \quad && D(A_2) \coloneqq \set{f \in L^1((-\infty, 1]) : ((x - 1) f)' \in L^1((-\infty, 1))},
\end{alignat*}
where the derivatives have to be understood in the distributional sense. 

These operators are closely related in the sense that both operators are similar via the \emph{reflection operator}
\begin{align*}
    Q \colon L^1([0, \infty)) \to L^1((-\infty, 1]), \quad Qf(x) \coloneqq f(1 - x).
\end{align*}
More precisely, one has the following lemma. We omit its proof as it is straightforward. 

\begin{lemma} \label{lemma:similarity-lemma-extended} \phantom{X}
\begin{enumerate}[label = \upshape (\roman*)]
    \item \label{lemma:similarity-lemma-extended-a1} The reflection operator $Q$ is an isometric isomorphism with
    \begin{align*}
        Q^{-1} \colon L^1((-\infty, 1]) \to L^1([0, \infty)), \quad Q^{-1}f(x) \coloneqq f(1 - x).
    \end{align*}
    In particular, both $Q$ and $Q^{-1}$ are positive.
    \item \label{lemma:similarity-lemma-extended-a2} The operators $A_1$ and $A_2$ are similar via $Q$, i.e., we have that $D(A_2) = Q D(A_1)$ and 
        \begin{align*}
            A_2 = Q A_1 Q^{-1}.
        \end{align*}
\end{enumerate}
\end{lemma}

Next, we prove that $A_1$ and $A_2$ generate stochastic semigroups on the spaces $L^1([0, \infty))$ and $L^1((-\infty, 1])$, respectively. For that purpose and for later reference, we give explicit identities for the resolvents of both operators.

\begin{lemma} \label{lemma:resolvent-formulas-extended} The following assertions hold:
	\begin{enumerate}[label = \upshape (\roman*)]
		\item \label{lemma:resolvent-formulas-extended-a1} We have $(0, \infty) \subseteq \rho(A_1)$. For all $\lambda > 0$ and $g \in L^1([0, \infty))$ the identity
		\begin{align*}
			R(\lambda, A_1) f(x) = x^{\lambda - 1} \int_x^\infty \frac{f(y)}{y^\lambda} \, \ud y
		\end{align*}
        holds for almost all $x \in [0, \infty)$. 
		\item \label{lemma:resolvent-formulas-extended-a2} We have $(0, \infty) \subseteq \rho(A_2)$. For all $\lambda > 0$ and $g \in L^\infty((-\infty, 1])$ the identity
		\begin{align*}
			R(\lambda, A_2) f(x) = (1 - x)^{\lambda - 1} \int_{-\infty}^x \frac{f(y)}{(1 - y)^\lambda} \, \ud y
		\end{align*}
        holds for almost all $x \in (-\infty,1]$. 
	\end{enumerate}
\end{lemma}

\begin{proof}
	(i): For $\lambda > 0$ we define the operator
	\begin{align*}
		R_\lambda \colon L^1([0, \infty)) \to L^1([0, \infty)), \quad (R_\lambda f)(x) \coloneqq x^{\lambda - 1} \int_x^\infty \frac{f(y)}{y^\lambda} \, \ud y.
	\end{align*}
    We show that $R_\lambda$ is well-defined, i.e., that $R_\lambda f \in L^1([0, \infty))$ for all $f \in L^1([0, \infty))$. Fix $f \in L^1([0, \infty))$. Clearly, $R_\lambda f$ is measurable for each $\lambda  > 0$ and by Fubini's theorem one has
	\begin{align*}
		\norm{R_\lambda f}_1 &= \int_0^\infty \abs{R_\lambda f(x)} \, \ud x \leq \int_0^\infty \int_x^\infty x^{\lambda - 1} \frac{\abs{f(y)}}{y^\lambda} \, \ud y \, \ud x \\ 
	&= \int_0^\infty \frac{\abs{f(y)}}{y^\lambda} \int_0^y x^{\lambda - 1} \, \ud x \, \ud y = \int_0^\infty \frac{\abs{f(y)}}{\lambda} \, \ud y = \frac{1}{\lambda} \norm{f}_1
	\end{align*}
and thus $R_\lambda f \in L^1([0, \infty))$. Furthermore, an elementary calculation shows that for all $\varphi \in \mathrm{C}_{\mathrm{c}}^\infty((0, \infty))$ we have
\begin{align*}
    \langle \varphi', x \cdot R_\lambda f \rangle = - \langle \varphi, \lambda R_\lambda f - f \rangle,
\end{align*}
which implies that
\begin{align*}
    \frac{\ud}{\ud x}(x \cdot R_\lambda f) = \lambda R_\lambda f - f \in L^1((0, \infty)),
\end{align*}
and thus, $R_\lambda f \in D(A_1)$ for all $\lambda > 0$. Furthermore, another standard calculation shows that
\begin{align*}
    (\lambda - A_1) R_\lambda f = R_\lambda (\lambda - A_1) f = f \qquad \text{for all } f \in D(A_1).
\end{align*}
Hence, the claim follows. 

(ii): Since $A_1$ and $A_2$ are similar by Lemma~\ref{lemma:similarity-lemma-extended}, it follows that $\sigma(A_1) = \sigma(A_2)$, and thus $(0, \infty) \subseteq \rho(A_2)$. Also by similarity, $R(\lambda, A_2) = Q R(\lambda, A_1) Q^{-1}$ (see e.g.\ \cite[Section~II.2.1]{Engel2000}) and thus
\begin{align*}
	R(\lambda, A_2) f(x) = (1 - x)^{\lambda - 1} \int_{1 - x}^\infty \frac{f(1 - y)}{y^\lambda} \, \ud y = (1 - x)^{\lambda - 1} \int_{-\infty}^x \frac{f(y)}{(1 - y)^\lambda} \, \ud y
\end{align*}
for all $\lambda > 0$ and $f \in L^1([0, \infty))$. 
\end{proof}

Now we are in a position to show that $A_1$ and $A_2$ generate stochastic semigroups on the spaces $L^1([0, \infty))$ and $L^1((-\infty, 1])$, respectively.

\begin{proposition} \label{proposition:generated-semigroups-extended}
The following assertions hold:
\begin{enumerate}[label = \upshape (\roman*)]
    \item \label{proposition:generated-semigroups-extended-a1} The operator $A_1$ is the generator of a stochastic $C_0$-semigroup $(\widetilde T_1(t))_{t \geq 0}$ on $L^1([0, \infty))$, given by
    \begin{align*}
        \widetilde T_1(t) f(x) = \ue^t f(x \ue^t)
    \end{align*}
    for all $f \in L^1([0, \infty))$.
    \item \label{proposition:generated-semigroups-extended-a2} The operator $A_2$ is the generator of a stochastic $C_0$-semigroup $(\widetilde T_2(t))_{t \geq 0}$ on $L^1((-\infty, 1])$, given by
    \begin{align*}
        \widetilde T_2(t) f(x) = \ue^t f(1 - (1 - x) \ue^t)
    \end{align*}
    for all $f \in L^1((-\infty, 1])$.
\end{enumerate}
\end{proposition}

\begin{proof}
    (i): It is elementary to show that $(\widetilde T_1(t))_{t \geq 0}$ satisfies the semigroup law and that it is strongly continuous. So, it is just left to show that $A_1$ is indeed the generator of $(\widetilde T_1(t))_{t \geq 0}$. To this end, let us denote the generator of $(\widetilde T_1(t))_{t \geq 0}$ by $A$. Let $f \in D(A_1)$. Then $\widetilde T_1(t) f \in D(A_1)$ and the classical Sobolev embedding implies that $T_1(t) f$ is continuous. In particular, we have
    \begin{align*}
        R(\lambda, A) f(x) &= \int_0^\infty \ue^{-\lambda t} \widetilde T_1(t) f(x) \, \ud t = \int_0^{\infty} \ue^{-\lambda t} \ue^t f(x \ue^t) \, \ud t \\
        &= x^{\lambda - 1} \int_x^\infty \frac{f(y)}{y^\lambda} \, \ud y = R(\lambda, A_1) f(x)
    \end{align*}
    by substitution for all $\lambda > 0$ and $x \in [0, \infty)$. Since $R(\lambda, A)$ and $R(\lambda, A_1)$ coincide on $D(A)$ and $D(A)$ is dense in $L^1([0, \infty))$, it follows that $A = A_1$.

     Finally, by using the same substitution as above, one obtains
     \begin{align*}
         \norm{\widetilde T_1(t) f}_1 = \norm{f}_1 \qquad \text{for all } 0 \leq f \in D(A_1),
     \end{align*}
    which implies by density of $D(A_1)$ that $(\widetilde T_1(t))_{t \geq 0}$ is a stochastic semigroup.

    (ii): The claim that $A_2$ generates a stochastic semigroup follows immediately from Lemma~\ref{lemma:similarity-lemma-extended}. Furthermore, Lemma~\ref{lemma:similarity-lemma-extended} implies
    \begin{align*}
        \widetilde T_2(t) f(x) = Q \widetilde T_1(t) Q^{-1} f(x) = \ue^t f(1 - (1 - x) \ue^t)
    \end{align*}
    for all $f \in L^1((-\infty, 1])$. 
\end{proof}

From the explicit formulas of the resolvents of $A_1$ and $A_2$ in Lemma~\ref{lemma:resolvent-formulas-extended} we now derive formulas for the resolvents of the dual operators. 

\begin{corollary} \label{corollary:dual-resolvent-identity-extended}
The following assertions hold:
	\begin{enumerate}[label = \upshape (\roman*)]
		\item \label{corollary:dual-resolvent-identity-extended-a1} We have $(0, \infty) \subseteq \rho(A_1')$ and for all $\lambda > 0$ and all $g \in L^\infty([0, \infty))$ that
		\begin{align*}
			R(\lambda, A_1') g(y) = \frac{1}{y^\lambda} \int_0^y g(x) x^{\lambda - 1} \, \ud x
		\end{align*}
        for almost all $y \in [0,\infty)$. 
		\item \label{corollary:dual-resolvent-identity-extended-a2} We have $(0, \infty) \subseteq \rho(A_2')$ and for all $\lambda > 0$ and all $g \in L^\infty((-\infty, 1])$ that 
		\begin{align*}
			R(\lambda, A_2') g(y) = \frac{1}{(1 - y)^\lambda} \int_y^1 g(x) (1 - x)^{\lambda - 1} \, \ud x
		\end{align*}
        for almost all $y \in (-\infty,1]$. 
	\end{enumerate}
\end{corollary}

\begin{proof}
	(i): For $\lambda > 0$ we define the operator
	\begin{align*}
		R_\lambda' \colon L^\infty([0, \infty)) \to L^\infty([0, \infty)), \quad (R_\lambda' f)(y) \coloneqq \frac{1}{y^\lambda} \int_0^y f(x) x^{\lambda - 1} \, \ud x.
	\end{align*}
    Then $R_\lambda'$ is well-defined with $\norm{R_\lambda'} = \frac 1 \lambda$. Now let $f \in L^1([0, \infty))$ and $g \in L^\infty([0, \infty))$. Then and Fubini's theorem yields
        \begin{align*}
	       \langle R(\lambda, A_1) f, g \rangle &= \int_0^\infty \int_x^\infty x^{\lambda - 1} \frac{f(y)}{y^\lambda}  g(x) \, \ud y \, \ud x \\
	       &= \int_0^\infty f(y) \frac{1}{y^\lambda} \int_0^y  g(x) x^{\lambda - 1} \, \ud x \, \ud y = \langle f, R_\lambda' g \rangle
        \end{align*}
    and, thus, $R(\lambda, A_1') = R(\lambda, A_1)' = R_\lambda '$ for all $\lambda > 0$.
    
	(ii): Since $A_1$ and $A_2$ are similar via $Q$, we have $R(\lambda, A_2') = Q'^{-1} R(\lambda, A_1') Q'$ for all $\lambda > 0$, where
    \begin{align*}
        Q' \colon L^\infty((-\infty, 1]) \to L^\infty([0, \infty)), \quad Q'g(x) \coloneqq g(1 - x)
    \end{align*}
    denotes the dual of the reflection operator $Q$. Thus, a simple calculation yields the claim. 
\end{proof}

Since the Theorem~\ref{theorem:main-result} is concerned with semigroups on the space $L^1([0, 1])$, one needs to restrict the semigroups $\tilde{T}_1$ and $\tilde{T}_2$ generated by $A_1$ and $A_2$, respectively, to this smaller space in a meaningful way. This is what is done next.

In order to study these restrictions of the semigroups to $L^1([0, 1])$, we consider the following obvious candidates for their generators. We define the operators $C_1$ and $C_2$ by
\begin{alignat*}{2}
    C_1 &\coloneqq \frac{\ud}{\ud x}(x f), \ \ && D(C_1) \coloneqq \set{f \in L^1([0, 1]) : (x f)' \in L^1((0, 1)), \, f(1) = 0}, \\
    C_2 &\coloneqq \frac{\ud}{\ud x}((x-1) f), \ \ && D(C_2) \coloneqq \set{f \in L^1([0, 1]) : ((x - 1) f)' \in L^1((0, 1)), \, f(0) = 0}.
\end{alignat*}

The following proposition collects some identities related to the operators $C_1$ and $C_2$. These are used later on to obtain lower bounds for the resolvent of the generator of the semigroup $(T(t))_{t \geq 0}$.

\begin{proposition} \label{proposition:generated-semigroups}
The following assertions hold:
\begin{enumerate}[label = \upshape (\roman*)]
    \item \label{proposition:generated-semigroups-a1} The operator $C_1$ is the generator of a stochastic $C_0$-semigroup $(T_1(t))_{t \geq 0}$ on $L^1([0, 1])$, given by
    \begin{align*}
        T_1(t) f(x) = 
        \begin{cases}
            \ue^t f(x \ue^t), \quad &\textup{if } t \leq - \ln x, \\
            0, \quad &\textup{else},
        \end{cases}
    \end{align*}
    for all $f \in L^1([0, 1])$. Moreover, for all $\lambda > 0$ and $f \in L^1([0, 1])$ the identity
    \begin{align*}
        R(\lambda, C_1) f(x) = x^{\lambda - 1} \int_x^1 \frac{f(y)}{y^\lambda} \, \ud y
    \end{align*}
    holds for almost all $x \in [0,1]$. 
    \item \label{proposition:generated-semigroups-a2} The operator $C_2$ is the generator of a stochastic $C_0$-semigroup $(T_2(t))_{t \geq 0}$ on $L^1([0, 1])$, given by
    \begin{align*}
        T_2(t) f(x) = 
        \begin{cases}
            \ue^t f(1 - (1 - x) \ue^t), \quad &\textup{if } t \leq - \ln (1 - x), \\
            0, \quad &\textup{else},
        \end{cases}
    \end{align*}
    for all $f \in L^1([0, 1])$. Moreover, for all $\lambda > 0$ and $f \in L^1([0, 1])$ the identity
    \begin{align*}
        R(\lambda, C_2) f(x) = (1 - x)^{\lambda - 1} \int_0^x \frac{f(y)}{(1 - y)^\lambda} \, \ud y
    \end{align*}
    holds for almost all $x \in [0,1]$.
\end{enumerate}   
\end{proposition}

\begin{proof}
(i): Consider the closed, $(\widetilde T_1(t))_{t \geq 0}$-invariant subspace
\begin{align*}
    F \coloneqq \set{f \in L^1([0, \infty)) : f(x) = 0 \text{ for almost all } x \geq 1}.
\end{align*}
Then the part $A_{1 \rvert}$ of $A_1$ in $F$ is given by
\begin{align*}
    A_{1 \rvert} f & = A_1 f = \frac{\ud}{\ud x}(xf),\\
    D(A_{1 \rvert})  &= \set{f \in D(A_1) \cap F : A_1 f \in F} \\
    & = \{ f \in L^1([0,\infty)) : (x f)' \in L^1([0,\infty)), \, f(x) = 0 \text{ for all } x \geq 1\}.
\end{align*}
and generates the subspace semigroup $(\widetilde T_1(t)_\rvert)_{t \geq 0}$ on $F$ (see \cite[Section~II.2.3]{Engel2000}). Clearly, $F$ is isomorphic $L^1([0, 1])$ and the generator of the subspace semigroup $(\widetilde T_1(t)_\rvert)_{t \geq 0}$ is isomorphic to $C_1$. The above resolvent identity is another direct consequence of this isomorphy. 

(ii): This follows either from the fact that $(T_2(t))_{t \geq 0}$ is similar to $(T_1(t))_{t \geq 0}$ via the restricted reflection operator
\begin{align*}
    Q \colon L^1([0, 1]) \to L^1([0, 1]), \quad Qf(x) = f(1 - x)
\end{align*}
which is an isometric isomorphism with $Q^{-1} = Q$ or, alternatively, by analogous arguments as used in (i).
\end{proof}

Analogously to Corollary~\ref{corollary:dual-resolvent-identity-extended}, the following result is a consequence of Proposition~\ref{proposition:generated-semigroups}.

\begin{corollary} \label{corollary:dual-resolvent-identity}
The following assertions hold:
	\begin{enumerate}[label = \upshape (\roman*)]
		\item \label{corollary:dual-resolvent-identity-a1} We have $(0, \infty) \subseteq \rho(C_1')$ and for all $\lambda > 0$ and all $g \in L^\infty([0,1])$ that
		\begin{align*}
			R(\lambda, C_1') g(y) = \frac{1}{y^\lambda} \int_0^y g(x) x^{\lambda - 1} \, \ud x
		\end{align*}
        for almost all $y \in [0,\infty)$.
		\item \label{corollary:dual-resolvent-identity-a2} We have $(0, \infty) \subseteq \rho(C_2')$ and for all $\lambda > 0$ and all $g \in L^\infty([0,1])$ that
		\begin{align*}
			R(\lambda, C_2') g(y) = \frac{1}{(1 - y)^\lambda} \int_y^1 g(x) (1 - x)^{\lambda - 1} \, \ud x
		\end{align*}
        for almost all $y \in [0,\infty)$. 
	\end{enumerate}
\end{corollary}

\section{Long-term behaviour} \label{section:the-long-term-behaviour}

In this section, we prove our main result (Theorem~\ref{theorem:main-result}). Namely, we show that the solutions to \eqref{eq:the-pde} converge uniformly to the equilibrium of the equation with respect to the operator norm. In Section~\ref{section:setting-the-stage} we observed that this is equivalent to studying the long-term behaviour of the solutions to the abstract Cauchy-problem
\begin{align} \tag{ACP}
    u'(t) = (\calA + \calB) u(t), \quad u(0) = u_0,
\end{align}
in the space $L^1([0, 1]; \C^2)$. Here $\calA$ and $\calB $ are given as at the beginning of Section~\ref{section:setting-the-stage}. Recall that by the discussion in the introduction of Section~\ref{section:setting-the-stage} we may, and shall, assume that $a = 0$ and $c = d$. It was already shown in \cite[Proposition~4.3]{Kurasov2021} that the operator $\calA + \calB$ generates a stochastic, irreducible semigroup $(T(t))_{t \geq 0}$ on $L^1([0, 1])$.

We first recall the notion of integral operators. Let $(\Omega, \Sigma, \mu)$ be a $\sigma$-finite measure space. Then a bounded operator $T$ on $L^1(\Omega)$ is called an \emph{integral operator} if there exists a measurable function $k \colon \Omega \times \Omega \to \R$ such that for each $f \in L^1(\Omega)$ the function $y \mapsto k(x, y) f(y)$ is in $L^1(\Omega)$ for almost every $x \in \Omega$ and
\begin{align*}
    T f = \int_\Omega k(\argument, y) f(y) \, \ud y \quad \text{for all } f \in L^1(\Omega).
\end{align*}
Furthermore, an operator $T$ on $L^1(\Omega)$ is called a \emph{partial integral operator} if there exists a non-zero integral operator $K \in \calL(L^1(\Omega))$ such that $0 \leq K \leq T$. The following result (see \cite[Corollary~1.2]{Glueck2022}) is our main tool for the proof of Theorem~\ref{theorem:main-result}.

\begin{theorem}\label{theorem:glueck-martin}
    Let $(\Omega, \Sigma, \mu)$ be a $\sigma$-finite measure space and let ${(T(t))}_{t \geq 0}$ be a stochastic and irreducible $C_0$-semigroup on $L^1(\Omega)$. Suppose that there exists $t_0 > 0$ such that $T(t_0)$ is a partial integral operator. Then the following statements are equivalent:
    \begin{enumerate}[label = \upshape (\roman*)]
        \item ${(T(t))}_{t \geq 0}$ converges with respect to the operator norm as $t \to \infty$.
        \item\label{item:prop:glueck_martin-property_pos_improv} For each non-zero $0 \leq g \in L^\infty(\Omega)$ there exists $\lambda > 0$ such that $R(\lambda, A)' g \gg 0$.
    \end{enumerate}
\end{theorem}

So to apply Theorem~\ref{theorem:glueck-martin} to the semigroup $(S(t))_{t \geq 0}$ generated by $\calA + \calB$, we need to show that the following two assertions hold:

\begin{enumerate}[label = \upshape (\alph*)]
    \item\label{item:to-show-a} For each non-zero $0 \leq g \in L^\infty([0, 1]; \C^2)$ there exists some $\lambda > 0$ and some $c > 0$ such that 
    \begin{align*}
        R(\lambda, \calA + \calB)' g \geq c \cdot \one.
    \end{align*}
    This is shown in Section~\ref{section:dual-resolvent-positivity-improving}.
    \item\label{item:to-show-b} There exists a time $t_0 > 0$ such that $S(t_0)$ is a partial integral operator. This is shown in Section~\ref{section:partial-integral-operator}. 
\end{enumerate}
Note that, as discussed at the beginning of Section~\ref{section:setting-the-stage}, the space $L^1([0, 1]; \C^2)$ is isometrically isomorphic to a scalar-valued $L^1$-space due to the choice of the norm. In particular, the above theorem is applicable.

\subsection{The dual resolvent improves positivity} \label{section:dual-resolvent-positivity-improving}

The next proposition shows that condition \ref{item:to-show-a} is true.

\begin{proposition} \label{proposition:dual-resolvent-positivity-improving}
    For each non-zero $0 \leq g \in L^\infty([0, 1]; \C^2)$ there exists some $\lambda > 0$ and some $c > 0$ such that $R(\lambda, \calA + \calB)' g \geq c \one$.
\end{proposition}

The following lemma simplifies estimates later in the proof of Proposition~\ref{proposition:dual-resolvent-positivity-improving}.

\begin{lemma} \label{lemma:lambda-0-epsilon}
There exists $\varepsilon > 0$, $\gamma > 0$ and some $\lambda_0 > 0$ such that the following assertions hold:
\begin{enumerate}[label = \upshape (\roman*)]
    \item $R(\lambda, \calA + \calB) \geq R(\lambda + \gamma, \calA + \calE)$ for all $\lambda > \lambda_0$. 
    \item $\norm{\calE R(\lambda, \calA)} < 1$ for all $\lambda > \lambda_0$.
\end{enumerate}
Here $\calE \in \calL(L^1([0, 1]; \C^2))$ is given by
\begin{align*}
    \calE f \coloneqq \varepsilon
    \begin{pmatrix}
        1 & 1 \\
        1 & 1
    \end{pmatrix} f \qquad \text{for all } f \in L^1([0, 1]; \C^2). 
\end{align*}
\end{lemma}

\begin{proof}
    Fix $\gamma > \max \set{\norm{\nu}_\infty, \norm{\mu}_\infty}$ and pick $\varepsilon > 0$ such that 
    \begin{align*}
        \varepsilon < \gamma - \max \set{\norm{\nu}_\infty, \norm{\mu}_\infty} \quad \text{and} \quad \varepsilon < \min \set[\bigg]{\min_{x \in [0, 1]} \nu(x), \min_{x \in [0, 1]} \mu(x)}. 
    \end{align*}
    Such an $\varepsilon$ exists, since $\nu$ and $\mu$ are strictly positive and continuous. Now pick $\lambda_0 > \max \set{\gamma, 2 \varepsilon}$ (we may later enlarge $\lambda_0$ further if necessary). We show that the assertions (i) and (ii) hold for our choices of $\varepsilon$ and $\lambda_0$.

    (i): By the choice of $\varepsilon$ and $\gamma$ one has
    \begin{align*}
        \calA + \calB + \gamma I = \calA + \calE + (\calB + \gamma I - \calE),
    \end{align*}
    where $\calB + \gamma I - \calE$ is a positive operator.
    Moreover, since $\calA$ is the generator of a stochastic semigroup and $\calE$ is a positive operator, it follows from the Trotter product formula $\ue^{t(\calA+ \calE)}f = \lim_{n \to \infty} \ue^{\tfrac{t}{n} \calA} \ue^{\tfrac{t}{n} \calE}f$ for all $f \in L^1((0,1);\bbC^2)$ (see, e.g.,~\cite[Corollary~III.5.8]{Engel2000}) that the semigroup generated by $\calA + \calE$ is positive. Hence, the integral representation of the resolvent implies that $R(\lambda, \calA + \calE)$ is positive for all $\lambda > \lambda_0$ with $\lambda_0$ large enough. Moreover, it follows from the Hille-Yosida theorem and by choosing $\lambda_0$ even larger that
    \begin{align*}
        \norm{(\calB + \gamma \id - \calE) R(\lambda, \calA + \calE)} & \leq \norm{(\calB + \gamma \id - \calE)} \norm{R(\lambda, \calA + \calE)} \\
        & \leq \norm{(\calB + \gamma \id - \calE)} \cdot \frac{1}{\lambda - \omega} < 1
    \end{align*}
    for all $\lambda > \lambda_0$, where $\omega \in \bbR$ is the growth bound of the semigroup generated by $\calA + \calE$. Thus, by a consequence of the Neumann series expansion (see, e.g., \cite[Theorem 4.5]{Rudnicki2017}), we obtain that
    \begin{align*}
       R(\lambda, \calA + \calB) & = R(\lambda + \gamma, \calA + \calB + \gamma I) = R(\lambda + \gamma, \calA + \calE + (\calB + \gamma I - \calE)) \\
       & = R(\lambda + \gamma, \calA + \calE) \sum_{k = 0}^\infty \big( (\calB + \gamma I - \calE) R(\lambda + \gamma, \calA + \calE) \big)^k \\
       & \geq R(\lambda + \gamma, \calA + \calE)
    \end{align*}
    for all $\lambda > \lambda_0$.
    
    (ii): Since $\calA$ is the generator of a stochastic semigroup, we have
    \begin{align*}
        \norm{\calE R(\lambda, \calA)}
        \leq \norm{\calE} \cdot \norm{R(\lambda, \calA)} 
        \leq \frac{2 \varepsilon}{\lambda} < \frac{2 \varepsilon}{\lambda_0} < 1
    \end{align*}
    for all $\lambda > \lambda_0$ by the Hille-Yosida theorem.
\end{proof}

We are now in the position to prove Proposition~\ref{proposition:dual-resolvent-positivity-improving}.

\begin{proof}[Proof of Proposition~\ref{proposition:dual-resolvent-positivity-improving}]
By Lemma~\ref{lemma:lambda-0-epsilon}, there exist $\varepsilon > 0$, $\gamma > 0$ and $\lambda_0 > \gamma$ such that 
\begin{align*}
   R(\lambda - \gamma, \calA + \calB) \geq R(\lambda, \calA + \calE),
\end{align*}
and thus,
\begin{align*}
    R(\lambda - \gamma, \calA + \calB)' \geq R(\lambda, \calA + \calE)',
\end{align*}
and $\norm{R(\lambda, \calA)' \calE'} = \norm{\calE R(\lambda, \calA)} < 1$ for all $\lambda > \lambda_0$.
We show that there exists $\lambda > \lambda_0$ such that $R(\lambda, \calA + \calE)' \gg 0$. So, by \cite[Theorem 4.5]{Rudnicki2017}, one has for $\lambda > \lambda_0$ that
\begin{align*}
    R(\lambda, \calA + \calB)' \geq R(\lambda, \calA + \calE)' = \sum_{k = 0}^\infty (R(\lambda, \calA)' \calE')^k R(\lambda, \calA)' \geq (R(\lambda, \calA)' \calE')^2 R(\lambda, \calA)'.
\end{align*}
Moreover, 
\begin{align*}
    R(\lambda, \calA)' \calE' &=  
    \begin{pmatrix}
        b^{-1} R(b^{-1} \lambda, C_1)' & 0 \\
        0 & d^{-1} R(d^{-1} \lambda, C_2)'
    \end{pmatrix}
    \begin{pmatrix}
        \varepsilon & \varepsilon \\
        \varepsilon & \varepsilon
    \end{pmatrix} \\
    &= \varepsilon
    \begin{pmatrix}
        b^{-1} R(b^{-1}\lambda, C_1)' & b^{-1} R(b^{-1}\lambda, C_1)' \\
        d^{-1} R(d^{-1}\lambda, C_2)' & d^{-1} R(d^{-1}\lambda, C_2)'
    \end{pmatrix}
\end{align*}
and hence
\begin{align*}
   (R(\lambda, \calA)' \calE')^2 &= 
    \varepsilon^2
    \begin{pmatrix}
        b^{-1} R(b^{-1}\lambda, C_1)' & b^{-1} R(b^{-1}\lambda, C_1)' \\
        d^{-1} R(d^{-1}\lambda, C_2)' & d^{-1} R(d^{-1}\lambda, C_2)'
    \end{pmatrix}^2 \\
    &\geq
    \frac{\varepsilon^2}{b d}
    \begin{pmatrix}
        R(b^{-1}\lambda, C_1)' R(d^{-1}\lambda, C_2)' & R(b^{-1}\lambda, C_1)' R(d^{-1}\lambda, C_2)' \\
        R(d^{-1}\lambda, C_2)' R(b^{-1}\lambda, C_1)' & R(d^{-1}\lambda, C_2)' R(b^{-1}\lambda, C_1)'
    \end{pmatrix} .
\end{align*}
So altogether, we have
\begin{align} \label{inq:dual-resolvent-estimate}
    R(\lambda, \calA + \calE)' g \geq \frac{\varepsilon^2}{b d}
    \begin{pmatrix}
        R(b^{-1}\lambda, C_1)' R(d^{-1}\lambda, C_2)' & R(b^{-1}\lambda, C_1)' R(d^{-1}\lambda, C_2)' \\
        R(d^{-1}\lambda, C_2)' R(b^{-1}\lambda, C_1)' & R(d^{-1}\lambda, C_2)' R(b^{-1}\lambda, C_1)'
    \end{pmatrix}
    \widetilde g,
\end{align}
where $\widetilde g \coloneqq R(\lambda, \calA)' g \geq 0$ is non-zero, whenever $0 \leq g \in L^\infty([0, 1]; \C^2)$ is non-zero. 

So to conclude the proof, it suffices to show that
\begin{align*}
    R(d^{-1}\lambda, C_2)' R(b^{-1}\lambda, C_1)' h \gg 0 \quad \text{and} \quad R(b^{-1}\lambda, C_1)' R(d^{-1}\lambda, C_2)' h \gg 0 
\end{align*}
for each non-zero $0 \leq h \in L^\infty([0,1])$. Let $0 \leq h \in L^\infty([0,1])$ be non-zero and choose $\lambda > \max \{ \lambda_0, b, d, \gamma \}$. Then, by Corollary~\ref{corollary:dual-resolvent-identity}, we have
\begin{align*}
    R(b^{-1}\lambda, C_1)' h(y) = \frac{1}{y^{\lambda/b}} \int_0^y h(x) x^{(\lambda - b)/b} \, \ud x.
\end{align*}
By \cite[Lemma~8.2]{Brezis2011}, there exists $\alpha \in (0, 1)$ and $\delta > 0$ such that
\begin{align*}
    R(b^{-1}\lambda, C_1)' h(y) \geq \frac{\delta}{y^{\lambda/b}} \geq \delta, \qquad \text{for all } y \in [\alpha, 1]. 
\end{align*}
Using Corollary~\ref{corollary:dual-resolvent-identity}, again, we obtain for all $z \in [0,1]$ that
\begin{align*}
    R(d^{-1}\lambda, C_2)' R(b^{-1}\lambda, C_1)' h(z) &= \frac{1}{(1 - z)^{\lambda/d}} \int_z^1 R(b^{- 1}\lambda, C_1)' h(y) (1 - y)^{\lambda/d - 1} \, \ud y \\
    &\geq \frac{1}{(1 - z)^{\lambda/d}} \int_{\alpha \vee z}^1 \delta (1 - y)^{\lambda/d - 1} \, \ud y  \\
    &= \frac{d \delta}{\lambda} \frac{(1 - (\alpha \vee z))^{\lambda/d}}{(1 - z)^{\lambda/d}} \geq \frac{d \delta}{\lambda} (1 - \alpha)^{\lambda/d}.
\end{align*}
So, all in all,
\begin{align*}
   R(d^{-1}\lambda, C_2)' R(b^{-1}\lambda, C_1)' h \geq d \delta (1 - \alpha)^{\lambda/d} \one \gg 0. 
\end{align*}
Completely analogously, one shows that $R(b^{-1}\lambda, C_1)' R(d^{-1}\lambda, C_2)' h \gg 0$. Thus, the claim follows as a consequence of the estimate \eqref{inq:dual-resolvent-estimate}.
\end{proof}

\subsection{Partial integral operator} \label{section:partial-integral-operator}

To apply Theorem~\ref{theorem:glueck-martin} it remains to show that the semigroup $(S(t))_{t \geq 0}$ generated by $\calA + \calB$ contains a partial integral operator, i.e., that condition~\ref{item:to-show-b} is satisfied.

\begin{proposition}\label{prop:partial-integral-operator}
    Let ${(S(t))}_{t \geq 0}$ be the semigroup generated by $\calA + \calB$. Then there exists $t_0 > 0$ such that $S(t_0)$ is a partial integral operator.
\end{proposition}

\begin{proof}
\begin{enumerate}[label = \emph{Step~\arabic*:}, wide, leftmargin = 0em, labelindent = 0em, itemsep = 1em]
    \item We show that there exist $\varepsilon > 0$ and $\gamma > 0$ such that the semigroup $(S(t))_{t \geq 0}$ generated by $\calA + \calB$ and the semigroup $(S_\varepsilon(t))_{t \geq 0}$ generated by $\calA + \calE$, where $\calE$ is defined as in Lemma~\ref{lemma:lambda-0-epsilon}, satisfy  $\ue^{\gamma t} S(t) \geq S_\varepsilon(t)$.

    Indeed, by Lemma~\ref{lemma:lambda-0-epsilon}, there exist $\varepsilon > 0$, $\gamma > 0$ and $\lambda_0 > 0$ such that $R(\lambda, \calA + \calB) \geq R(\lambda + \gamma, \calA + \calE) = R(\lambda, \calA + \calE - \gamma)$ for all $\lambda > \lambda_0$. Hence, the Post--Widder inversion formula (see~\cite[Corollary III.5.5]{Engel2000}) yields
    \begin{align}
        S(t) \geq \ue^{-\gamma t} S_\varepsilon(t) \qquad \text{for all } t \geq 0.
    \end{align}

    \item We use a perturbation argument to obtain a lower bound for the stochastic semigroup $(T(t))_{t \geq 0}$ generated by $\calA$ (cf. Proposition~\ref{proposition:generated-semigroups}): As $\calE$ is a bounded operator, the operator $\calA_\varepsilon \coloneqq \calA + \calE$ generates a $C_0$-semigroup $(S_\varepsilon(t))_{t \geq 0}$ and this semigroup is given by the Dyson--Phillips series 
    \begin{align*}
        S_\varepsilon(t) &= \sum_{k = 0}^\infty U_k(t), \qquad \text{for all } t \geq 0
    \end{align*}
    (see, e.g., \cite[Theorem~III.1.10]{Engel2000} or \cite[Section~11.2]{Batkai2017}), where
    \begin{align*}
        U_0(t) \coloneqq T(t), \quad U_k(t) &\coloneqq \int_0^t T(t - s) \calE U_{k - 1}(s) \, \ud s \qquad \text{for all } t \geq 0, \, k \in \bbN.
    \end{align*}
    Clearly, $S_\varepsilon(t) \geq U_1(t)$ for all $t \geq 0$.

    \item  Next, we decompose $U_1(t)$ into operators, which we later show to be partial integral operator for some $t > 0$. This then implies that $U_1(t)$ is a partial integral operator, which in turn implies that $S_\varepsilon(t)$, and thus, $S(t)$ are partial integral operators.
    
    Recall from Proposition~\ref{proposition:generated-semigroups} the definition of $(T_1(t))_{t \geq 0}$ and $(T_2(t))_{t \geq 0}$.
    Clearly,
\begin{align*}
    U_1(t) f &= 
    \int_0^t
    \begin{pmatrix}
        T_1(b(t-s)) & 0 \\
        0 & T_2 (d(t-s))
    \end{pmatrix}
    \begin{pmatrix}
        \varepsilon & \varepsilon \\
        \varepsilon & \varepsilon
    \end{pmatrix}
    \begin{pmatrix}
        T_1(bs) & 0 \\
        0 & T_2 (ds)
    \end{pmatrix} f \, \ud s \\
    &= \int_0^t \varepsilon
    \begin{pmatrix}
        T_1(b(t-s)) & T_1(b(t-s)) \\
        T_2(d(t-s)) & T_2 (d(t-s))
    \end{pmatrix}
    \begin{pmatrix}
        T_1(bs) & 0 \\
        0 & T_2 (ds)
    \end{pmatrix} f \, \ud s \\
    &= \int_0^t \varepsilon
    \begin{pmatrix}
        T_1(b(t-s))T_1(bs) & T_1(b(t-s))T_2(ds) \\
        T_2(d(t-s)) T_1(bs) & T_2 (d(t-s))T_2(bs)
    \end{pmatrix} f \, \ud s,
\end{align*}
dominates both functions
\begin{align*}
        \int_0^t \varepsilon
    \begin{pmatrix}
        0 & T_1(b(t-s))T_2(ds) \\
        0 & 0
    \end{pmatrix} f \, \ud s \quad \text{and} \quad \int_0^t \varepsilon
    \begin{pmatrix}
        0 & 0 \\
        T_2(d(t-s)) T_1(bs) & 0
    \end{pmatrix} f \, \ud s
\end{align*}
for all $f \in L^1([0, 1]; \C^2)$. Thus, it is sufficient to show that either
\begin{align} \label{eq:scalar-integral-operator-case-1}
   \int_0^t T_1(b(t-s))T_2(ds) \, \ud s \in \calL(L^1([0,1]))
\end{align}
or 
\begin{align} \label{eq:scalar-integral-operator-case-2}
   \int_0^t T_2(b(t-s))T_1(ds) \, \ud s \in \calL(L^1([0,1]))
\end{align}
is a partial integral operator.

\item We show that~\eqref{eq:scalar-integral-operator-case-1} is a partial integral operator if $d \leq b$. Let $f \in L^1([0,1])$ and $t \geq 0$. Then, by Proposition~\ref{proposition:generated-semigroups}, we have
\begin{align*}
    \int_0^t T_1(b(t-s)) T_2(d s)  f(x) \, \ud s
    = \int_0^t \ue^{b(t-s)} \ue^{ds} f(1- (1 - x \ue^{b(t-s)}) \ue^{ds} ) \, \ud s
\end{align*}
for almost all $x \in [0, 1]$. As $0 < d \leq b$, the mapping
\begin{align*}
    \varphi \colon [0,1] \times [0, 1] \to \bbR, \quad (x,s) \mapsto 1- (1 - x \ue^{b(t-s)}) \ue^{ds}.
\end{align*}
has the partial derivative
\begin{align*}
    \frac{\partial}{\partial s} \varphi (x, s) = d \ue^{d s} + (b-d)x \ue^{b(t-s)} \ue^{ds} > 0.
\end{align*}
Thus, $\varphi(x, \argument)$ is strictly increasing for each $x \in [0, 1]$. A substitution yields
\begin{align*}
    \int_0^t T_1(b(t-s)) T_2(d s)  f(x) \, \ud s &= \int_0^t \ue^{b(t-s)} \ue^{ds} f(\varphi(x, s)) \, \ud s \\
    &\geq \int_0^t f(\varphi(x, s)) \, \ud s \\
    &= \int_{\varphi(x, 0)}^{\varphi(x, t)}\frac{1}{\frac{\partial}{\partial s} \varphi (x, s)} f(s) \, \ud s.
\end{align*}
Hence, $\int_0^{t} T_1(b(t-s))T_2(ds) \, \ud s$ is a partial integral operator for each $t > 0$.

\item If $b \leq d$, then the proof showing that~\eqref{eq:scalar-integral-operator-case-2} is a partial integral operator for each $t > 0$ is analogous to the proof presented in the previous step.
\end{enumerate}

In total, we have proven that $U_1(t_0)$, and thus $S(t_0)$, is a partial integral operator for some $t_0 > 0$.
\end{proof}

\subsection{Proof of the main result}

We have now gathered all the tools to conclude with the proof of Theorem~\ref{theorem:main-result}.
\begin{proof}[Proof of Theorem~\ref{theorem:main-result}]
    Without loss of generality, we may assume that $a = 0$ and $c = d$. By \cite[Proposition~4.3]{Kurasov2021}, the operator $\calA + \calB$ generates an irreducible stochastic semigroup $(S(t))_{t \geq 0}$ on $L^1([0, 1]; \C^2)$. Proposition~\ref{prop:partial-integral-operator} shows that there exists a time $t_0 > 0$ such that $S(t_0)$ is a partial integral operator. Moreover, by Proposition~\ref{proposition:dual-resolvent-positivity-improving}, for each non-zero $0 \leq g \in L^\infty([0, 1]; \C^2)$ there exists some $\lambda > 0$ such that $R(\lambda, \calA + \calB)' g \gg 0$. So, Theorem~\ref{theorem:glueck-martin} implies that $(S(t))_{t \geq 0}$ converges with respect to the operator norm as $t \to \infty$.
\end{proof}

\subsubsection*{Acknowledgements} The first-named author thanks Jochen Glück for an invitation to the University of Wuppertal, where major parts of this paper have been written.

\subsection*{Conflict of interest} The authors declare to have no conflict of interest. 

\subsection*{Data availability statement} Due to the nature of the research, there is no supporting data. 

\bibliographystyle{plain}

\bibliography{literature}

\begin{thebibliography}{10}

\bibitem{Aliprantis2006}
Charalambos~D. Aliprantis and Kim~C. Border.
\newblock {\em Infinite dimensional analysis: A hitchhiker's guide}.
\newblock Springer, Berlin, third edition, 2006.

\bibitem{Aliprantis1985}
Charalambos~D. Aliprantis and Owen Burkinshaw.
\newblock {\em Positive operators}, volume 119 of {\em Pure and Applied
  Mathematics}.
\newblock Academic Press, Inc., Orlando, FL, 1985.

\bibitem{Arendt2008}
Wolfgang Arendt.
\newblock Positive semigroups of kernel operators.
\newblock {\em Positivity}, 12(1):25--44, 2008.

\bibitem{Arendt2020}
Wolfgang Arendt and Jochen Gl\"{u}ck.
\newblock Positive irreducible semigroups and their long-time behaviour.
\newblock {\em Philos. Trans. Roy. Soc. A}, 378(2185):20190611, 17, 2020.

\bibitem{Batkai2017}
Andr\'{a}s B\'{a}tkai, Marjeta Kramar~Fijav\v{z}, and Abdelaziz Rhandi.
\newblock {\em Positive operator semigroups: From finite to infinite
  dimensions}, volume 257 of {\em Operator Theory: Advances and Applications}.
\newblock Birkh\"{a}user/Springer, Cham, 2017.

\bibitem{Brezis2011}
Haim Brezis.
\newblock {\em Functional analysis, {S}obolev spaces and partial differential
  equations}.
\newblock Universitext. Springer, New York, 2011.

\bibitem{Davies2005}
E.~B. Davies.
\newblock Triviality of the peripheral point spectrum.
\newblock {\em J. Evol. Equ.}, 5(3):407--415, 2005.

\bibitem{Engel2000}
Klaus-Jochen Engel and Rainer Nagel.
\newblock {\em One-parameter semigroups for linear evolution equations}, volume
  194 of {\em Graduate Texts in Mathematics}.
\newblock Springer-Verlag, New York, 2000.
\newblock With contributions by S. Brendle, M. Campiti, T. Hahn, G. Metafune,
  G. Nickel, D. Pallara, C. Perazzoli, A. Rhandi, S. Romanelli and R.
  Schnaubelt.

\bibitem{Gerlach2013}
Moritz Gerlach.
\newblock On the peripheral point spectrum and the asymptotic behavior of
  irreducible semigroups of {Harris} operators.
\newblock {\em Positivity}, 17(3):875--898, 2013.

\bibitem{GerlachDISS}
Moritz Gerlach.
\newblock {\em Semigroups of kernel operators}.
\newblock PhD thesis, Universit{\"a}t Ulm, 2014.

\bibitem{Gerlach2017}
Moritz Gerlach and Jochen Gl\"{u}ck.
\newblock On a convergence theorem for semigroups of positive integral
  operators.
\newblock {\em C. R. Math. Acad. Sci. Paris}, 355(9):973--976, 2017.

\bibitem{Gerlach2019}
Moritz Gerlach and Jochen Gl\"{u}ck.
\newblock Convergence of positive operator semigroups.
\newblock {\em Trans. Amer. Math. Soc.}, 372(9):6603--6627, 2019.

\bibitem{Glueck2019}
Jochen Gl\"{u}ck and Markus Haase.
\newblock Asymptotics of operator semigroups via the semigroup at infinity.
\newblock In {\em Positivity and noncommutative analysis}, Trends Math., pages
  167--203. Birkh\"{a}user/Springer, Cham, 2019.

\bibitem{Glueck2022}
Jochen Gl\"{u}ck and Florian~G. Martin.
\newblock Uniform convergence of stochastic semigroups.
\newblock {\em Israel J. Math.}, 247(1):1--19, 2022.

\bibitem{Greiner1982}
G\"{u}nther Greiner.
\newblock {\em Spektrum und {A}symptotik stark stetiger {H}albgruppen positiver
  {O}peratoren}, volume~82 of {\em Sitzungsberichte der Heidelberger Akademie
  der Wissenschaften. Mathematisch-Naturwissenschaftliche Klasse [Reports of
  the Heidelberg Academy of Science. Section for Mathematics and Natural
  Sciences]}.
\newblock Springer-Verlag, Berlin, 1982.

\bibitem{Jahnke2006}
Tobias Jahnke and Wilhelm Huisinga.
\newblock Solving the chemical master equation for monomolecular reaction
  systems analytically.
\newblock {\em Journal of Mathematical Biology}, 54(1):1--26, September 2006.

\bibitem{Kurasov2018}
Pavel Kurasov, Alexander L\"{u}ck, Delio Mugnolo, and Verena Wolf.
\newblock Stochastic hybrid models of gene regulatory networks---a {PDE}
  approach.
\newblock {\em Math. Biosci.}, 305:170--177, 2018.

\bibitem{Kurasov2021}
Pavel Kurasov, Delio Mugnolo, and Verena Wolf.
\newblock Analytic solutions for stochastic hybrid models of gene regulatory
  networks.
\newblock {\em J. Math. Biol.}, 82(1-2):Paper No. 9, 29, 2021.

\bibitem{Laurenzi2000}
Ian~J. Laurenzi.
\newblock An analytical solution of the stochastic master equation for
  reversible bimolecular reaction kinetics.
\newblock {\em The Journal of Chemical Physics}, 113(8):3315--3322, 2000.

\bibitem{MeyerNieberg1991}
Peter Meyer-Nieberg.
\newblock {\em Banach lattices}.
\newblock Universitext. Springer-Verlag, Berlin, 1991.

\bibitem{Nagel1986}
R.~Nagel, editor.
\newblock {\em One-parameter semigroups of positive operators}, volume 1184 of
  {\em Lecture Notes in Mathematics}.
\newblock Springer-Verlag, Berlin, 1986.

\bibitem{Pichor2000}
Katarzyna Pich\'{o}r and Ryszard Rudnicki.
\newblock Continuous {M}arkov semigroups and stability of transport equations.
\newblock {\em J. Math. Anal. Appl.}, 249(2):668--685, 2000.

\bibitem{Pichor2016}
Katarzyna Pich\'{o}r and Ryszard Rudnicki.
\newblock Asymptotic decomposition of substochastic operators and semigroups.
\newblock {\em J. Math. Anal. Appl.}, 436(1):305--321, 2016.

\bibitem{Pichor2018a}
Katarzyna Pich\'{o}r and Ryszard Rudnicki.
\newblock Asymptotic decomposition of substochastic semigroups and
  applications.
\newblock {\em Stoch. Dyn.}, 18(1):1850001, 18, 2018.

\bibitem{Pichor2018}
Katarzyna Pich\'{o}r and Ryszard Rudnicki.
\newblock Stability of stochastic semigroups and applications to {S}tein's
  neuronal model.
\newblock {\em Discrete Contin. Dyn. Syst. Ser. B}, 23(1):377--385, 2018.

\bibitem{Pichor2012}
Katarzyna Pich{\'o}r, Ryszard Rudnicki, and Marta Tyran-Kami{\'n}ska.
\newblock Stochastic semigroups and their applications to biological models.
\newblock {\em Demonstr. Math.}, 45(2):463--494, 2012.

\bibitem{Puchalka2004}
Jacek Puchałka and Andrzej~M. Kierzek.
\newblock {B}ridging the {G}ap between {S}tochastic and {D}eterministic
  {R}egimes in the {K}inetic {S}imulations of the {B}iochemical {R}eaction
  {N}etworks.
\newblock {\em Biophysical Journal}, 86(3):1357--1372, 2004.

\bibitem{Rudnicki2002}
R.~Rudnicki, K.~Pich{\'o}r, and M.~Tyran-Kami{\'n}ska.
\newblock Markov semigroups and their applications.
\newblock In {\em Dynamics of dissipation. Invited lectures delivered during
  the 38th Karpacz winter school of theoretical physics on ``Dynamical
  semigroups: Dissipation, chaos, quanta'', L\k{a}dek Zdr\'oj, Poland, February
  6--15, 2002}, pages 215--238. Berlin: Springer, 2002.

\bibitem{Rudnicki2017}
Ryszard Rudnicki and Marta Tyran-Kami\'{n}ska.
\newblock {\em Piecewise Deterministic Processes in Biological Models}.
\newblock Springer-Briefs in Applied Sciences and Technology. Springer, Cham,
  2017.
\newblock SpringerBriefs in Mathematical Methods.

\bibitem{Schaefer1974}
Helmut~H. Schaefer.
\newblock {\em Banach lattices and positive operators}.
\newblock Die Grundlehren der mathematischen Wissenschaften, Band 215.
  Springer-Verlag, New York-Heidelberg, 1974.

\bibitem{Schnoerr2017}
David Schnoerr, Guido Sanguinetti, and Ramon Grima.
\newblock Approximation and inference methods for stochastic biochemical
  kinetics—a tutorial review.
\newblock {\em Journal of Physics A: Mathematical and Theoretical},
  50(9):093001, 2017.

\bibitem{Zaanen1997}
Adriaan~C. Zaanen.
\newblock {\em Introduction to operator theory in {R}iesz spaces}.
\newblock Springer-Verlag, Berlin, 1997.

\end{thebibliography}

\end{document}